\documentclass{amsart} 

\usepackage{amsmath,amssymb,amsthm} 
\usepackage{graphicx} 
\usepackage{subfig}
\usepackage[arrow, matrix, curve]{xy} 
\usepackage{xcolor} 
\usepackage{makecell} 
\usepackage{multirow} 
\usepackage{enumerate}
\usepackage{booktabs} 
\usepackage{siunitx} 
\usepackage{hyperref} 

\newcommand{\Z}{\mathbb{Z}}

\makeatletter
\newcommand{\Vast}{\bBigg@{2.5}} 
\makeatother

\makeatletter
\newcommand{\smallbullet}{} 
\DeclareRobustCommand\smallbullet{%
 \mathord{\mathpalette\smallbullet@{0.7}}
}
\newcommand{\smallbullet@}[2]{%
 \vcenter{\hbox{\scalebox{#2}{$\m@th#1\bullet$}}}%
}
\makeatother

\newtheoremstyle{thm}{}{}{\itshape}{}{\bfseries}{}{ }{} 
\newtheoremstyle{definition}{}{}{}{}{\bfseries}{}{ }{} 

\theoremstyle{thm}
\newtheorem{Theorem}{Theorem}[section]
\newtheorem{thm}[Theorem]{Theorem}
\newtheorem{lem}[Theorem]{Lemma}

\newtheorem*{Theorem-ohne}{Theorem}

\newtheorem*{thm:main_thm}{Theorem~\ref{thm:main_thm}}

\theoremstyle{definition}
\newtheorem{defi}[Theorem]{Definition}

\begingroup\expandafter\expandafter\expandafter\endgroup
\expandafter\ifx\csname pdfsuppresswarningpagegroup\endcsname\relax
\else
\fi

\definecolor{amaranth}{rgb}{0.9, 0.17, 0.31} 
\definecolor{carrotorange}{rgb}{0.93, 0.57, 0.13} 
\definecolor{citrine}{rgb}{0.89, 0.82, 0.04} 
\definecolor{dartmouthgreen}{rgb}{0.05, 0.5, 0.06} 
\definecolor{ballblue}{rgb}{0.13, 0.67, 0.8} 
\definecolor{ceruleanblue}{rgb}{0.16, 0.32, 0.75} 
\definecolor{amethyst}{rgb}{0.6, 0.4, 0.8} 
\definecolor{amber}{rgb}{1.0, 0.75, 0.0} 
\definecolor{burlywood}{rgb}{0.87, 0.72, 0.53} 

\numberwithin{equation}{section}

\begin{document}

\title{Links have no characterising slopes} 

\author{Marc Kegel}
\address{Universidad de Sevilla, Dpto.\ de Álgebra,
Avda.\ Reina Mercedes s/n,
41012 Sevilla}
\email{mkegel@us.es, kegelmarc87@gmail.com}

\author{Misha Schmalian}
\address{University of Oxford, Andrew Wiles Building, OX2 6GG, UK}
\email{schmalian@maths.ox.ac.uk}



\begin{abstract}
We show that there is no analogue of characterising slopes for multi-component links. Concretely, we show that for any ordered link $L$ in $S^3$ with $n\geq 2$ components and any slopes $r_1, \ldots, r_n \in \mathbb Q\cup\{\infty\}$, there are infinitely many links $L_i$ with non-homeomorphic complements such that the Dehn fillings $L(r_1, \ldots, r_n)$ and $L_i(r_1, \ldots, r_n)$ are homeomorphic.
\end{abstract}


\makeatletter
\@namedef{subjclassname@2020}{%
 \textup{2020} Mathematics Subject Classification}
\makeatother



\maketitle
\section{Introduction}
\noindent
Given a knot $K$ in $S^3$ let $K(r)$ denote the Dehn filling of the knot exterior $S^3\setminus N(K)$ along the slope $r\in \mathbb Q\cup \{\infty\}$ with respect to the Seifert framing. We say a slope $r$ is {\it characterising} for a knot $K$ if there is no other knot $K'$ such that $K(r)$ and $K'(r)$ are orientation-preservingly homeomorphic. Characterising slopes for knots have gathered substantial attention in recent years. For example, it is known that every slope is characterising for the unknot, trefoil, and figure-eight knot \cite{KMOS_unknot_char}, \cite{trefoil_eight_char}. Moreover, it is known that for any knot $K$ every slope $p/q$ with $|q|$ sufficiently large is characterising \cite{Lackenby_char}, \cite{McCoy_char_torus}, \cite{Sorya_satellite_char}. On the other hand, it is known that for any $r\in \mathbb Q\cup\{\infty\}$, there is some knot $K$ with $r$ not characterising for $K$ \cite{Wakelin_Picirillo_Hayden_any_slope}. For further results on characterising slopes of knots we refer to~\cite{KS_unique} and the references therein. \\

\noindent
In this article, we consider multi-component links and show that the analogue of a characterising slope does not exist in this setting. Let $L$ be an ordered, $n$-component link in $S^3$ and $(r_1, \ldots, r_n)\in (\mathbb Q\cup\{\infty\})^n$ be a multi-slope, then we denote by $L(r_1, \ldots, r_n)$ the Dehn surgery of $L$ along the multi-slope $(r_1, \ldots, r_n)$, i.e.\ for every $i=1,\ldots n$, we perform a surgery with slope $r_i$ on the $i$-th component of $L$. 

\begin{thm}\label{thm:main_thm}
For any ordered, $n$-component link $L$ in $S^3$ with $n\geq 2$ and any multi-slope $(r_1, \ldots, r_n), r_j\in \mathbb Q\cup\{\infty\}$, there are infinitely many $n$-component links $(L_i), i\in \mathbb N$ such that the Dehn surgeries $L(r_1, \ldots, r_n)$ and $L_i(r_1, \ldots, r_n)$ are all orientation-preservingly homeomorphic, but the complements $S^3\setminus L_i$ are pairwise non-homeomorphic (irrespective of orientation). 
\end{thm}

\subsection*{Conventions}
\noindent
\begin{itemize}
\item[-] All homeomorphisms are assumed to be orientation-preserving and are denoted by $\cong$. 
\item[-] For an ordered link $L$, we denote by $L(*, r, *, \ldots, *)$ the manifold obtained by Dehn filling the exterior $S^3 \setminus N(L)$ of $L$ with slope $r$ in the second boundary torus and leaving the remaining boundary tori unfilled. 
\item[-] We measure surgery coefficients of links with respect to the meridian and Seifert longitude of the separate components. We write $\mu_i$ and $\lambda_i$ for the meridian and Seifert longitude of the $i$-th component of $L$. Given orientations on the components of $L$, we fix orientations of $\mu_i, \lambda_i$ as follows: Orient the $\lambda_i$ to agree with the orientation on $L$. Orient $\mu_i$ by the curled fingers on a right hand gripping the link component with extended thumb pointing in the orientation of the link.
\end{itemize}

\subsection*{Acknowledgements}
MS would like to thank Marc Lackenby for his continued guidance and Colin McCulloch for helpful conversations. MK is supported by the DFG, German Research Foundation (Project: 561898308); by a Ram\'on y Cajal grant (RYC2023-043251-I) and by the project PID2024-157173NB-I00 funded by MCIN/AEI/10.13039/501100011033, ESF+ and FEDER, EU; and by a VII Plan Propio de Investigación y Transferencia (SOL2025-36103) of the University of Sevilla.

\section{Background - rational handle slides}
\noindent
Handle slides are usually discussed in the context of integer surgery diagrams. 
Let $K_1$ be a knot in the complement of an $n$-framed knot $K_2$, for an integer $n\in\Z$. 
Then $n$-framed longitude $n\mu_2 + \lambda_2$ of $K_2$ bounds a disk in $K_2(n)$, 
and hence $K_1$ is isotopic to a band connected sum of $K_1$ with $n\mu_2 + \lambda_2$; 
see for example~\cite{GS} for a detailed discussion. 
The same argument applies when $K_2$ is equipped with a non-integer surgery coefficient, see, 
for instance~\cite{CEK} for a discussion in the Legendrian case. 
However, since we could not find an explicit reference describing how the framing of $K_1$
changes under a rational handle slide, we include an explanation below. \\

\noindent
For that, we consider the following setup.
\begin{itemize}
\item[-] Let $L$ be a link in $S^3$ and let $K_1$ and $K_2$ denote its first two components; 
\item[-] Let $b$ be a band from $K_1$ to $N(K_2)$ - that is $b$ is an embedding of a strip $[0,1]\times [0,1]$ into $S^3$ disjoint from $L\cup N(K_2)$ apart from having $b(\{0\}\times [0,1])\subset K_1$ and $b(\{1\}\times [0,1])\subset \partial N(K_2)$. See Figure \ref{fig:handle_slide_def} on the left; 
\item[-] Let $C_{p_2, q_2}(K_2)$ be a curve of slope $r_2=p_2/q_2\in \mathbb Q$ in $\partial N(K_2)$ with respect to the Seifert framing. Up to isotopy, we can assume that $C_{p_2,q_2}(K_2)$ contains $b(\{1\}\times [0,1])$. 
\end{itemize}

\begin{defi}\label{def:handle_slide} Under the above setup, we obtain the knot $K'_1$ as the symmetric difference of $K_1, \partial(\text{Im}(b))$, and $C_{p_2, q_2}(K_2)$. For reference, consult Figure \ref{fig:handle_slide_def}. Call $K'_1$ the {\it handle slide} of $K_1$ along the band $b$ with slope $r_2$. Let $L'$ be the link $L$ with $K_1$ replaced by $K'_1$. We also call $L'$ the {\it handle slide} of $L$ along the band $b$. 
\end{defi}

\begin{figure}[htbp]
 \centering
 \includegraphics[width=0.7\linewidth]{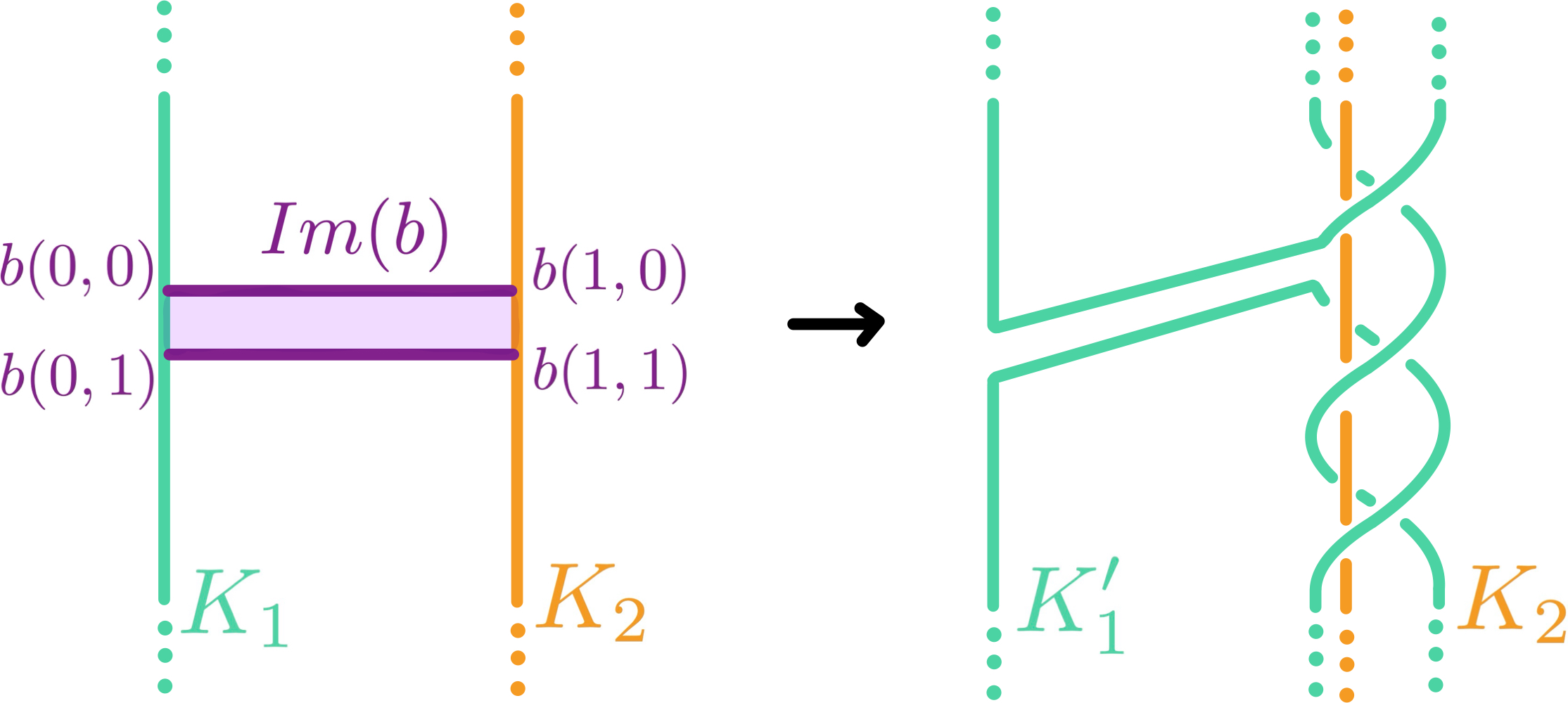} 
 \caption{The handle slide along a band $b$ with slope $3/2$.}
 \label{fig:handle_slide_def}
\end{figure}

\noindent
In the Dehn filling $L(*, r_2, *, \ldots, *)$ the curves $C_{p_2,q_2}(K_2)$ and $\partial(\text{Im}(b))$ bound discs, so in the Dehn filling $K_2(r_2)$ the images of $K_1$ and $K'_1$ are isotopic. 
Hence, we get a preferred homeomorphism $\varphi\colon L(*, r_2, *, \ldots, *)\to L'(*, r_2, *, \ldots, *)$. However, $\varphi$ might not respect the framings $\mu_i, \lambda_i$ on $L(*, r_2, *, \ldots, *)$ and $\mu'_i, \lambda'_i$ on $L'(*, r_2, *, \ldots, *)$ inherited from the Seifert framings of the components of $L$ and $L'$. Indeed, the framing change behaves as follows. 

\begin{lem}\label{lemma:framing_under_handle_slide}
 Given an $n$-component link $L$ with $n\geq 2$, consider the handle slide $L'$ of $L$ along the band $b$ with slope $r_2=p_2/q_2\in \mathbb Q$ as in the setup above. Then for any $r_1, r_3, \ldots, r_n\in \mathbb Q$ we have $L(r_1, r_2, \ldots, r_n)\cong L'(r_1+x, r_2, \ldots, r_n)$ for $$x=q_2\cdot p_2+2\cdot q_2\cdot \textrm{\emph{lk}}(K_1, K_2)),$$
 where to define $\text{\emph{lk}}(K_1, K_2)$, we orient $K_1, K_2$ in a manner agreeing with an orientation of $\partial \text{\emph{Im}}(b)$.
\end{lem}
\begin{proof} We adopt the notation for $\varphi, \mu_i, \lambda_i, \mu'_i, \lambda'_i$ as above. It is sufficient to show that $\varphi(\mu_i)=\mu'_i, \varphi(\lambda_i)=\lambda'_i$ for $i\geq 2$, $\varphi(\mu_1)=\mu'_1$ and $\varphi(\lambda_1)=x\cdot \mu'_1+\lambda'_1$. Since the handle slide construction of $L'$ leaves all but the first component $K_1$ of $L$ unchanged, we see that $\varphi(\mu_i)=\mu'_i, \varphi(\lambda_i)=\lambda'_i$ for $i\geq 2$. Moreover, the images of $K_1$ and $K'_1$ are isotopic in $K_2(r_2)$, so $\varphi$ respects meridians and $\varphi(\mu_i)=\mu'_i$. Thus $\varphi(\lambda_1)$ gets mapped to a curve of the form $\varphi(\lambda_1)=c\cdot \mu'_1+\lambda'_1$ for some $c\in \mathbb Z$ and it remains to show $c=x$. \\ 

\noindent
Let $\Sigma_1\subset S^3 \setminus N(K_1)$ be a Seifert surface for $K_1$. Let $\Sigma_{\text{cab}}\subset S^3\setminus N(C_{p_2,q_2}(K_2))$ be a Seifert surface for $C_{p_2,q_2}(K_2)$. Then, as shown in Figure \ref{fig:Seifert_surface}, the union 
$$\Sigma_1':=\Sigma_1\cup \text{Im}(b)\cup \Sigma_{\text{cab}}\subset S^3\setminus N(K'_1)$$ 
is an immersed oriented surface with boundary on $\partial N(K'_1)$ and in particular the homology class of $\partial \Sigma_1'$ is a multiple of the longitude $\lambda_1'$. Moreover, we see that there is a copy of the meridian $\mu_1$ intersecting $\partial \Sigma'_1$ once positively. Hence $\partial \Sigma'_1$ and the longitude $\lambda'_1$ are homologous. Observe that $\partial \Sigma_1'$ is the union of the following oriented curves. 
\begin{itemize}
 \item[-] The curve $\gamma'$ which is the symmetric difference of the longitude $\lambda_1$, the band boundary $\partial \text{Im}(b)$, and the longitude $\lambda_{cab}$ of $C_{p_2, q_2}(K_2)$. Since there is an orientation of $\partial \text{Im}(b)$ agreeing with the orientations of $K_1, K_2$, there is a natural choice of orientation for $\gamma'$.
 \item[-] Some number $c_1^+$ of positively oriented meridians $\mu'_1$ and $c_1^-$ of reversely oriented meridians $-\mu'_1$ with $c_1^+-c_1^-=-q_2\cdot \text{lk}(K_1, K_2)$. These occur as the intersection of $K_1$ with $\Sigma_{cab}$. 
 \item[-] Some number $c_2^+$ of positively oriented meridians $\mu'_1$ and $c_2^-$ of reversely oriented meridians $-\mu'_1$ with $c_2^+-c_2^-=-q_2\cdot \text{lk}(K_1, K_2)$. These occur as the intersection of $C_{p_2, q_2}(K_2)$ with $\Sigma_1$. 
 \item[-] Some number $c_3^+$ of meridians with standard orientation and $c_3^-$ of meridians with reversed orientation occuring as the intersection of $b([0,1]\times \{0,1\})$ with $\Sigma_1$ and $\Sigma_\text{cab}$; 
 \item[-] Finally, since $b([0,1]\times \{1\})$ has orientation opposite of $b([0,1]\times \{0\})$, we have $c_3^+=c_3^-$.
\end{itemize}
In summary, we have $$\lambda'_1=\gamma'-2\cdot q_2\cdot \text{lk}(K_1, K_2)\cdot \mu'_1\in H_1(\partial N(K'_1); \mathbb Z).$$

\begin{figure}[htbp]
 \centering
 \includegraphics[width=0.4\linewidth]{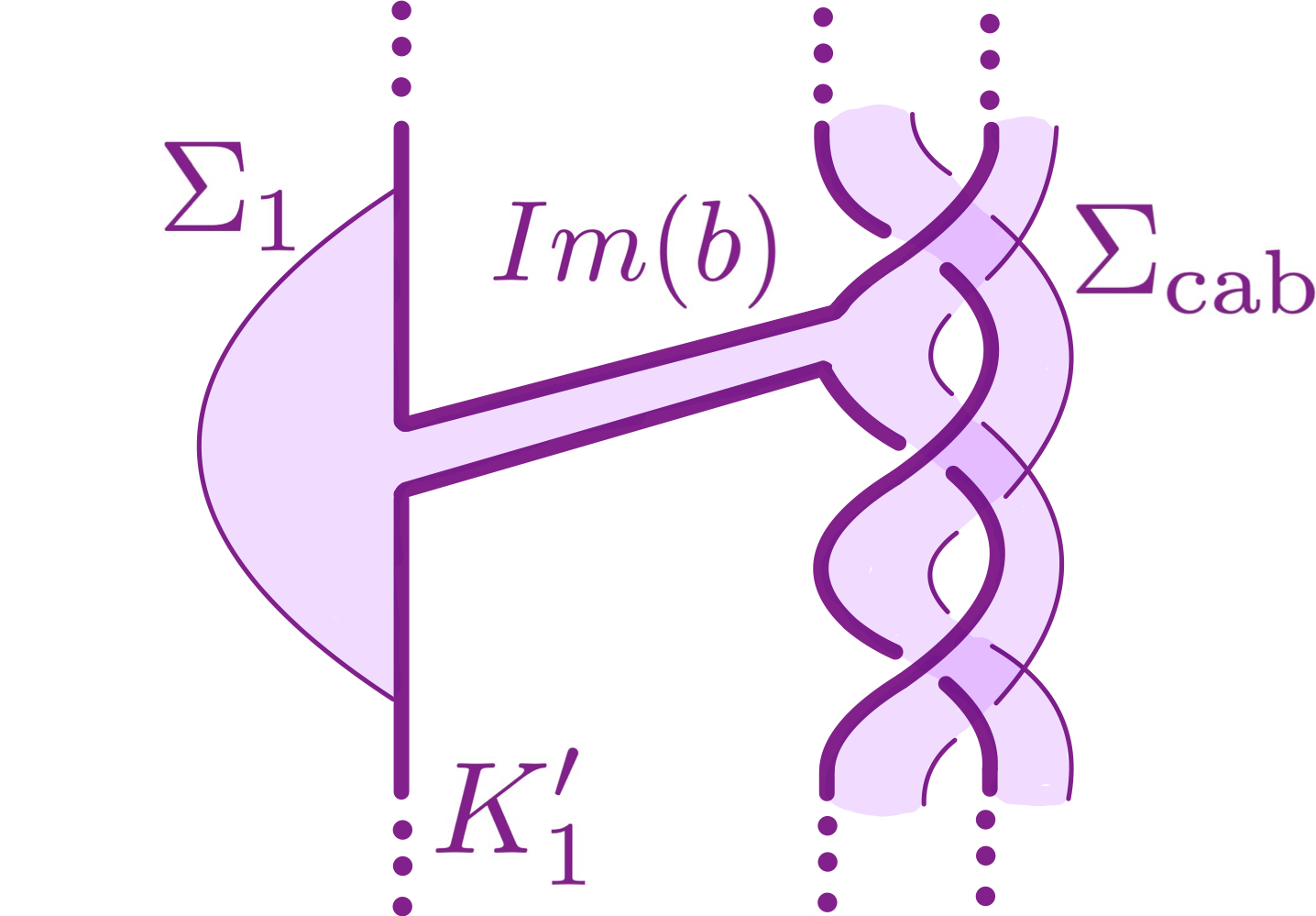} 
 \caption{A Seifert surface $\Sigma'_1$ of $K'_1$.}
 \label{fig:Seifert_surface}
\end{figure}

\noindent
Let $D_{r_2}\subset K_2(r_2)\setminus N(C_{p_2, q_2}(K_2))$ be a meridional disc of the Dehn filling with boundary on $\partial N(C_{p_2, q_2}(K_2))$. Then $\varphi(\lambda_1)$ is the symmetric difference of $\lambda_1$, $\partial\text{Im}(b)$, and $\partial(D_{r_2})$ with orientation inherited from $\lambda_1$.\\

\noindent
Observe that $i(\mu'_1, \varphi(\lambda_1))=i(\mu'_1, \gamma')=1$. Therefore 
$$\varphi(\lambda_1)=\gamma'+i(\varphi(\lambda_1), \gamma')\cdot \mu'_1=\lambda_1'+i(\varphi(\lambda_1), \gamma')\cdot \mu'_1+2\cdot q_2\cdot \text{lk}(K_1, K_2)\cdot \mu'_1.$$ So it remains to show that $i(\varphi(\lambda_1), \gamma')=q_2\cdot p_2$. Observe that $\varphi(\lambda_1)$ and $\gamma'$ agree on the $K_1$ and $\partial \text{Im}(b)$ portions of $K'_1$ and so $i(\varphi(\lambda_1), \gamma')=i(\partial D_{r_2}, \partial \Sigma_{\text{cab}})$, where $\partial D_{r_2}$ and $\partial \Sigma_{\text{cab}}$ are oriented to agree with the orientation of $C_{p_2,q_2}(K_2)$, which in turn is oriented to agree with the orientation of $K_2$.\\

\begin{figure}[htbp]
 \centering
 \includegraphics[width=0.5\linewidth]{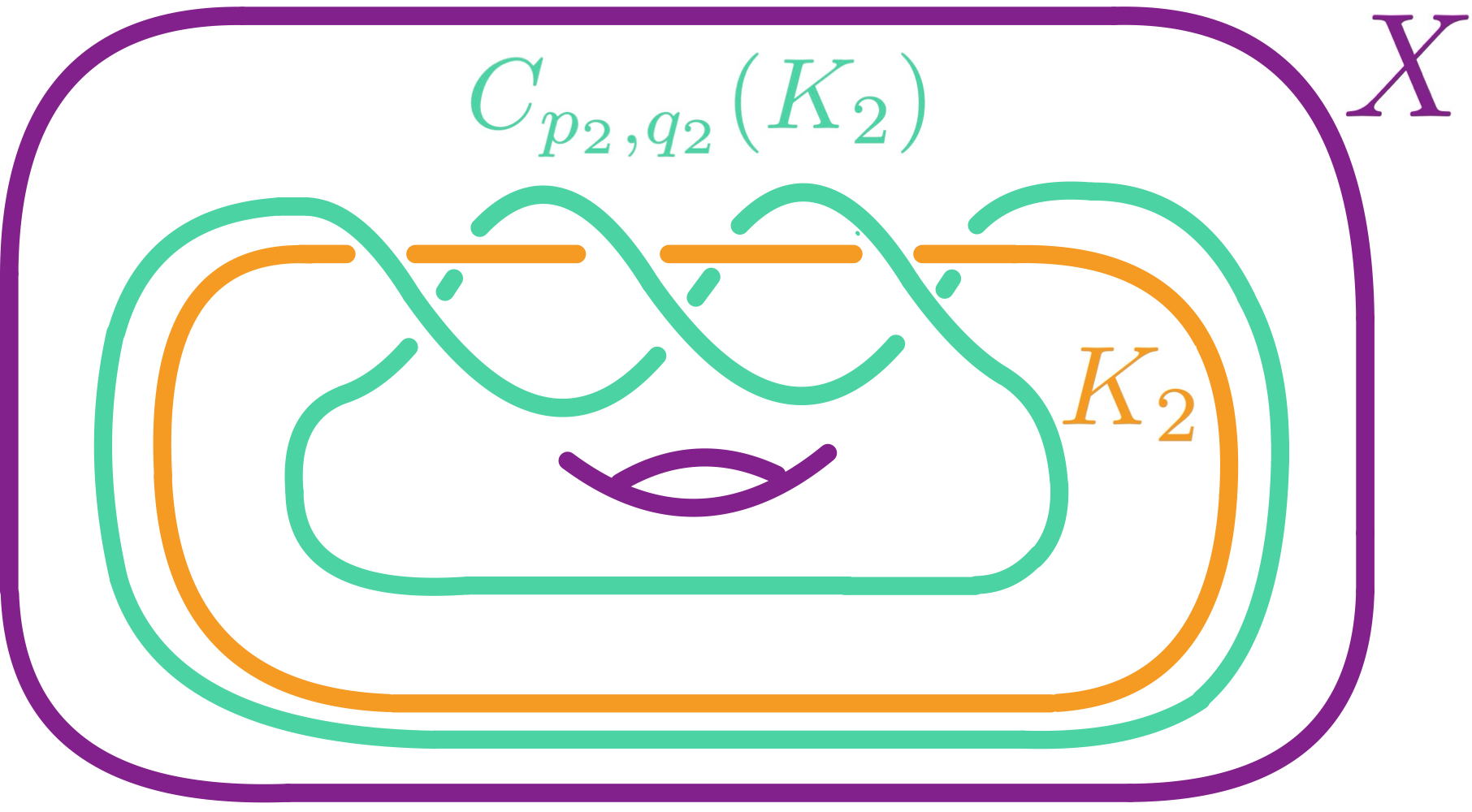} 
 \caption{The neighbourhood $X$ of $C_{p_2,q_2}(K_2)$ and $K_2$.}
 \label{fig:torus_nbhd_X}
\end{figure}

\noindent
Consider a solid torus neighbourhood $X\subset S^3$ of $C_{p_2,q_2}(K_2)$ with core curve $K_2$ as shown in Figure \ref{fig:torus_nbhd_X}. Equip $\partial X$ with the Seifert framing $\mu_X, \lambda_X$ of $K_2$. Then $\partial D_{r_2}$ is homotopic to $p_2\cdot \mu_X+q_2\cdot \lambda_X$. It is known, see for example \cite[Lemma 2.3]{Sorya_satellite_char}, that $\partial \Sigma_{\text{cab}}$ is homologous in $X\setminus N(C_{p_2,q_2}(K_2))$ to $q_2\cdot \lambda_{X}$. Let $\mu_{\text{cab}}$ be the meridian of $C_{p_2, q_2}(K_2)$. Then $\mu_X=q_2\cdot \mu_{\text{cab}}\in H_1(X\setminus N(C_{p_2,q_2}(K_2)); \mathbb Z)$. Hence 
$$\partial \Sigma_{\text{cab}}=q_2\cdot \lambda_X=\partial D_{r_2}-p_2\cdot \mu_X=\partial D_{r_2}-p_2\cdot q_2\cdot \mu_{\text{cab}}$$ 
and 
\begin{equation*}
 i(\partial D_{r_2}, \partial \Sigma_{\text{cab}})=i(\partial D_{r_2}, \partial D_{r_2}-p_2\cdot q_2\cdot \mu_{\text{cab}})=p_2\cdot q_2.\qedhere
\end{equation*}
\end{proof}

\section{Proof of Theorem~\ref{thm:main_thm}}
\noindent
With our understanding of the framing change under rational handle slides, we are ready to prove our main result.

\begin{proof}[Proof of Theorem~\ref{thm:main_thm}]
\noindent
 Suppose $r_1=\infty$. Then $L(r_1, \ldots, r_n)\cong J(r_1, \ldots, r_n)$ where $J$ is obtained from $L$ by replacing the component $K_1$ of $L$ with an arbitrary knot $K_2$. There are infinitely many links $J$ with pairwise non-homeomorphic complements that may be obtained this way. Hence, we may assume none of the $r_1, \ldots, r_n$ are $\infty$. \\

 \begin{figure}[htbp]
 \centering
 \includegraphics[width=0.5\linewidth]{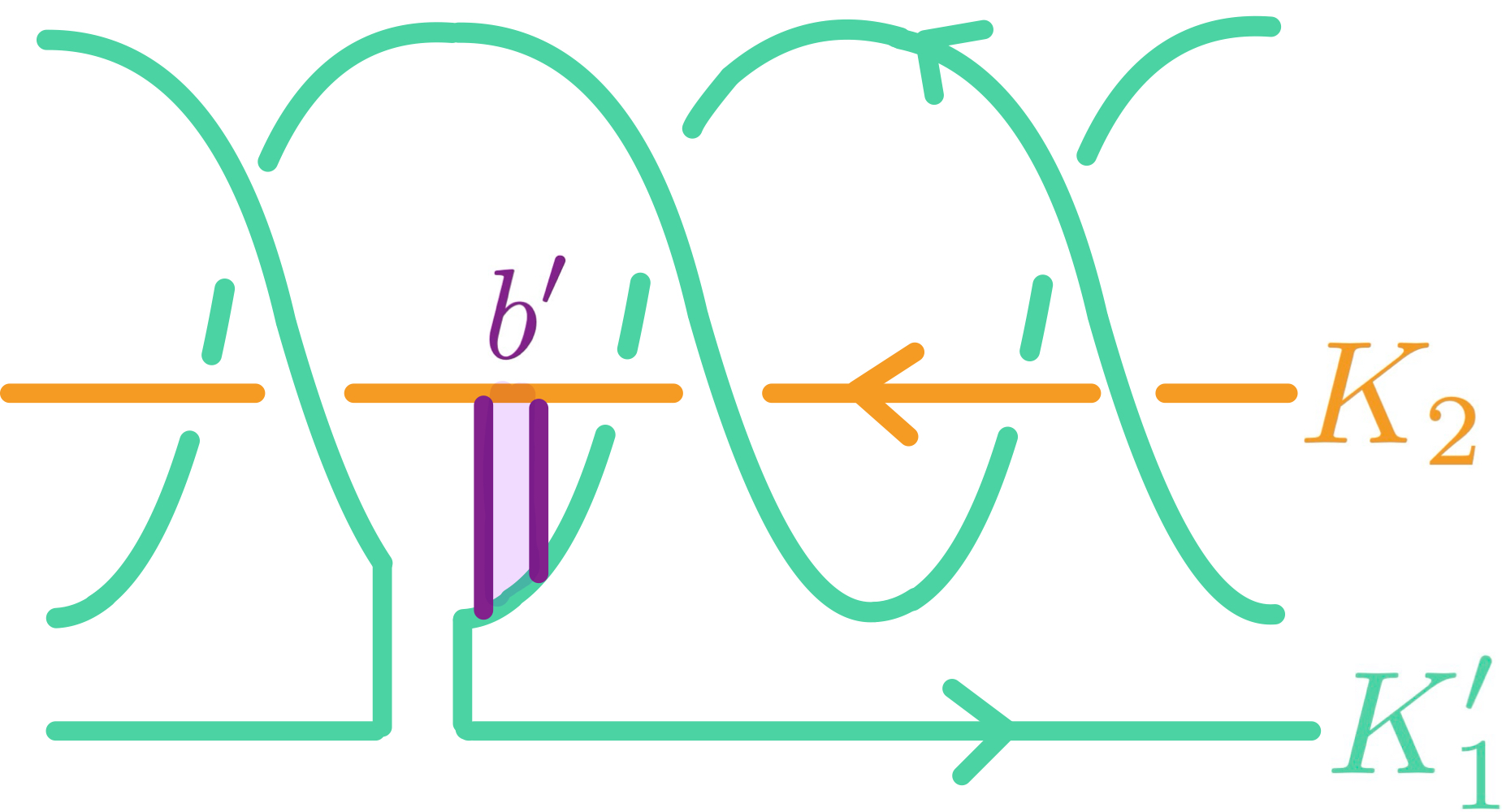} 
 \caption{The band $b'$ from $K'_1$ to $K_2$.}
 \label{fig:cancelling_handle_slide}
 \end{figure}

 \noindent
 Let $b:[0,1]^2\to S^3\setminus N(L)$ be a band from $K_1$ to $K_2$. Let the link $L'$ be the handle slide of $L$ along the band $b$ and slope $r_2$. Let $K'_1$ and $K_2$ denote the first two components of $L'$. 
 Consider a band $b'\colon[0,1]^2\to S^3\setminus N(L')$ from $K'_1$ to $K_2$ as shown in Figure \ref{fig:cancelling_handle_slide}. Finally, let the link $L''$, shown in Figure \ref{fig:L''}, be the handle slide of $L'$ along the band $b'$ with slope $r_2$. \\

 \begin{figure}[htbp]
 \centering
 \includegraphics[width=0.58\linewidth]{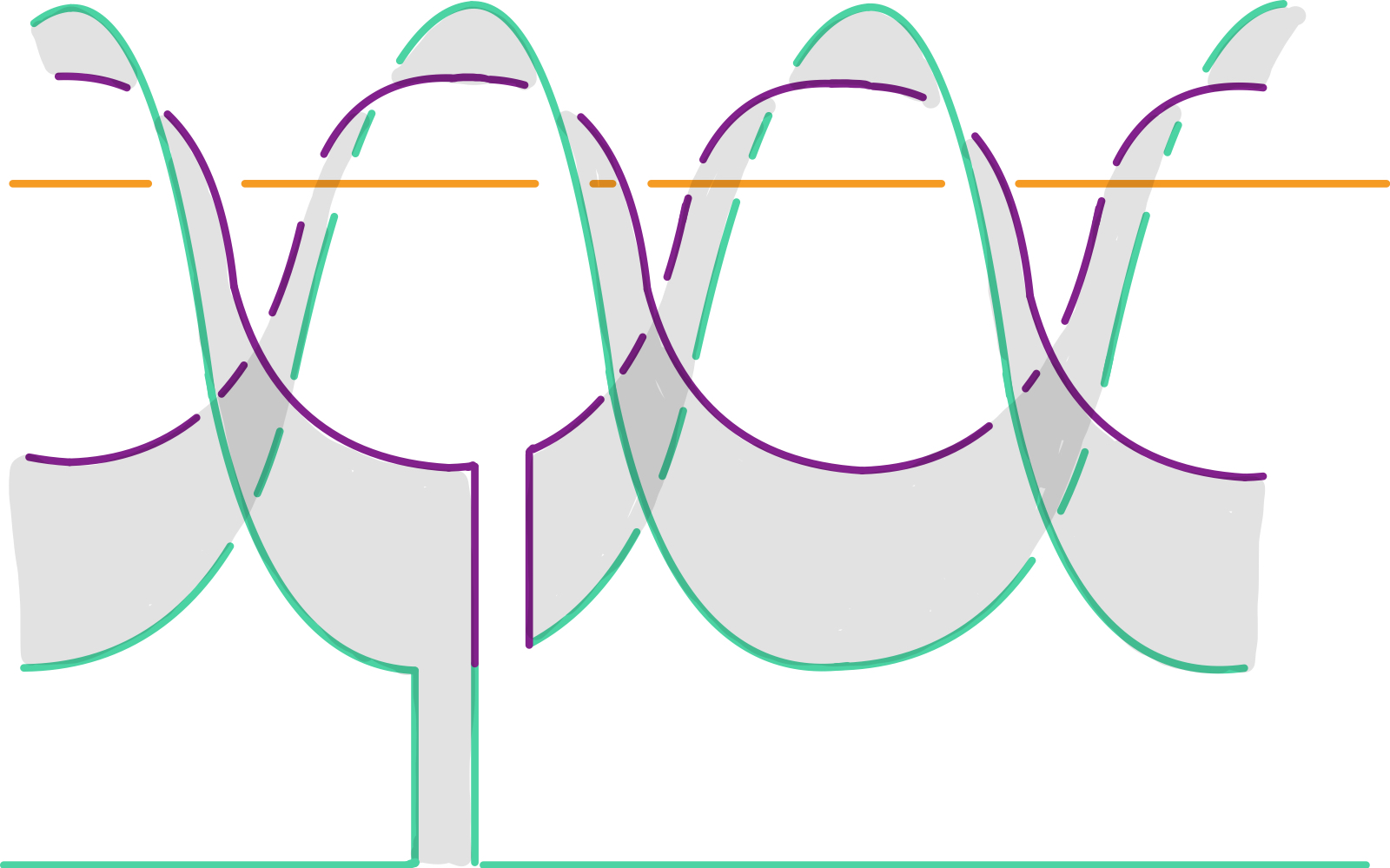} 
 \caption{The link $L''$.}
 \label{fig:L''}
 \end{figure}

 \noindent
 By contracting the shaded disc in Figure \ref{fig:L''}, we see that $L''$ is isotopic to $L$. In other words, the second handle slide along the band $b'$ reverses the first handle slide. Using Lemma \ref{lemma:framing_under_handle_slide} we can also compute that the surgery coefficients of $L$ and $L''$ agree. Indeed, we can orient $K_1, K_2$, and $K'_1$ as shown in Figure \ref{fig:cancelling_handle_slide}. In particular, the orientations of $K_1$ and $K_2$ agree with a choice of orientation of $\partial \text{Im}(b)$. Moreover, the orientation of $K'_1$ and the reverse $-K_2$ agree with a choice of orientation of $\partial \text{Im}(b')$. Thus, by Lemma \ref{lemma:framing_under_handle_slide}, we see that for any $r_1, r_3, \ldots, r_n\in \mathbb Q$, we have $$L(r_1, r_2, \ldots r_n)\cong L'(r_1+x, r_2, \ldots, r_n)\cong L''(r_1+x+y, r_2, \ldots, r_n),$$
 for $$x=q_2\cdot p_2+2\cdot q_2\cdot \text{lk}(K_1, K_2), \, \textrm{ and }~~~y=q_2\cdot p_2+2\cdot q_2\cdot \text{lk}(K'_1, -K_2).$$
 Observe that $\text{lk}(K'_1, -K_2)=-\text{lk}(K_1', K_2)=-(\text{lk}(K_1, K_2)+ p_2)$, and therefore
 $$x+y=2\cdot q_2\cdot p_2+2\cdot q_2\cdot \text{lk}(K_1, K_2)-2\cdot q_2\cdot (\text{lk}(K_1, K_2)+ p_2)=0.$$
 So $L(r_1, r_2, \ldots, r_n)\cong L''(r_1, r_2, \ldots, r_n)$ as claimed. \\

 \begin{figure}[htbp]
 \centering
 \includegraphics[width=0.55\linewidth]{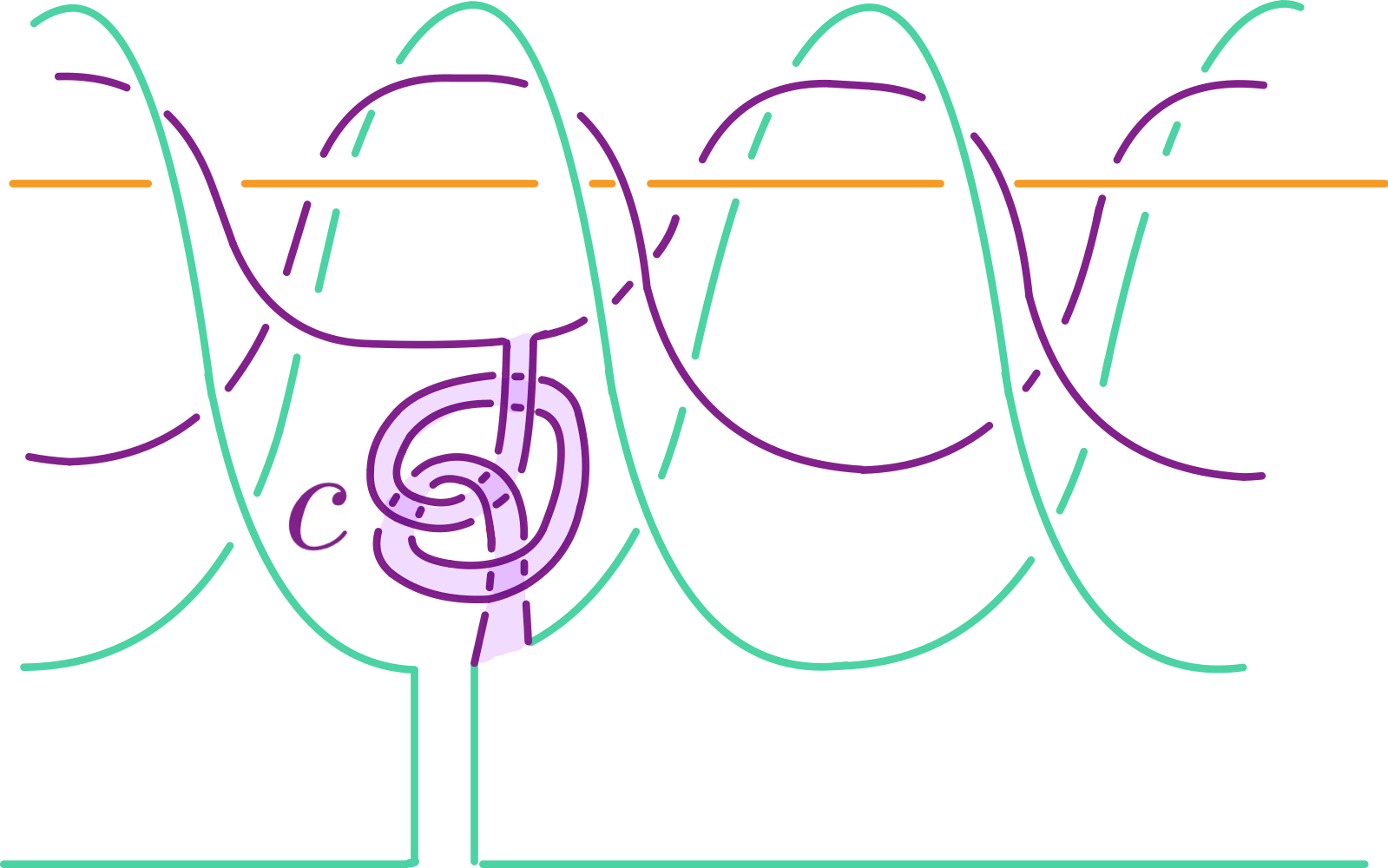} 
 \caption{The link $J$ with the band $c$ shaded in pink.}
 \label{fig:J}
 \end{figure}

\noindent
 Now let $c\colon[0,1]\times [0,1]\to S^3\setminus N(L')$ be a different band from $K'_1$ to $K_2$, such that $c$ and $b'$ are equal on $\{0,1\}\times [0,1]$ and let $J$, as shown in Figure \ref{fig:J}, be the handle slide of $L'$ along $c$ with slope $r_2$. Since the surgery coefficients of $L''$ and $J$ only depend on the homological information of the bands $b, b'$ and $c$, we conclude that the surgery coefficients of $J$ and $L''$ and thus $L$ agree. More concretely we compute as above using Lemma~\ref{lemma:framing_under_handle_slide} that $$L(r_1, r_2, \ldots, r_n)\cong L'(r_1+x, r_2, \ldots, r_n)\cong J(r_1+x+y, r_2, \ldots, r_n)\cong J(r_1, r_2, \ldots, r_n).$$ 

 \begin{figure}[htbp]
 \centering
 \includegraphics[width=0.5\linewidth]{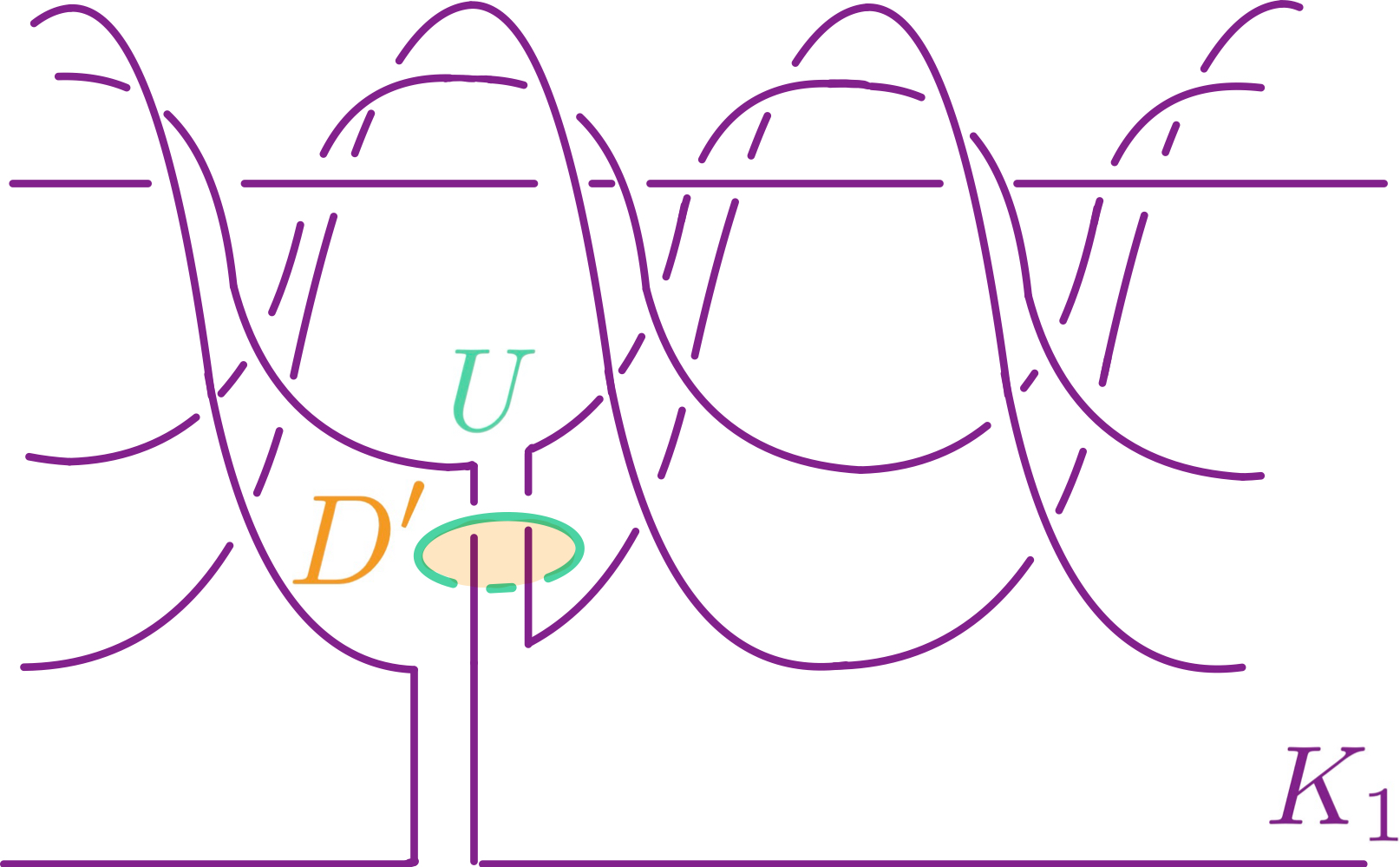} 
 \caption{The link $L$ in purple and the unknot $U$ in green.}
 \label{fig:U}
 \end{figure}

 \noindent
 It remains to show that, for a suitable choice of the bands $b, c$, and up to possibly swapping $K_1, K_2$, the links $J$ and $L$ have non-homeomorphic complements and that we can find an infinite family of such knots. To this end, we consider the unknot $U\subset S^3\setminus L$ circling the band $b'$ as shown in Figure \ref{fig:U}. Then $J$ is a satellite link with pattern $S^3\setminus N(U\cup L)$ and arbitrary companion knot $C$ determined by the band $c$. In other words, $S^3\setminus N(J)$ is obtained by gluing the boundary torus $T_U:=\partial(S^3\setminus N(U))$ of $S^3\setminus N(U\cup L)$ to the boundary torus of $S^3\setminus N(C)$. \\

 \noindent
 Suppose $T_U$ is incompressible in $S^3\setminus N(U\cup L)$. By \cite{Simp_Vol_Additive}, simplicial volume is additive under gluing along incompressible tori. Hence, if $C$ is a hyperbolic knot, then $S^3\setminus N(J)$ has larger simplicial volume than $S^3\setminus N(L)$ and, in particular, the two link exteriors are not homeomorphic. To get an infinite family of links $L_i$ that share the same surgery with $L$, we can take a sequence of hyperbolic knots $C_i$ with strictly increasing volumes, then the above construction give the desired sequence $L_i$, whose complements are distinguished by their simplicial volumes. In summary, to conclude the proof, let us show that $U$ does not bound a disc in $S^3\setminus N(L)$. \\

 \noindent
 Case 1. $K_1\cup K_2$ is not the unlink or $p_2, q_2$ are both not $1$:\\ Suppose there is a disc $D\subset S^3$ disjoint from $K_1$ with $\partial D=U$. Consider also the disc $D'\subset S^3$ with $\partial D'=U$ and with $D'$ intersecting the band $b'$ as shown in Figure \ref{fig:U}. Consider the immersed sphere $D\cup D'$ intersecting $K_1$ in two points $x_1, x_2$. Suppose $D\cup D'$ is not embedded and consider an innermost curve $\gamma$ of self-intersection. If $\gamma$ separates $x_1$ and $x_2$, then we obtain an immersed sphere in $S^3$ intersected by $K_1$ exactly once. In particular, this sphere pairs non-trivially with the homology class $[K_1]\in H_1(S^3)$ - a contradiction. Therefore $x_1$ and $x_2$ lie on the same side of $\gamma$ and by the irreducibility of $S^3\setminus K_1$, we may isotope $D$ relative to $K_1$ to remove $\gamma$. Hence we may assume $D\cup D'$ is an embedded sphere and that therefore $K_1$ is a composite knot with summands $K'_1$ and the cable knot $C_{p_2,q_2}(K_2)$. \\

 \begin{figure}[htbp]
 \centering
\includegraphics[width=0.5\linewidth]{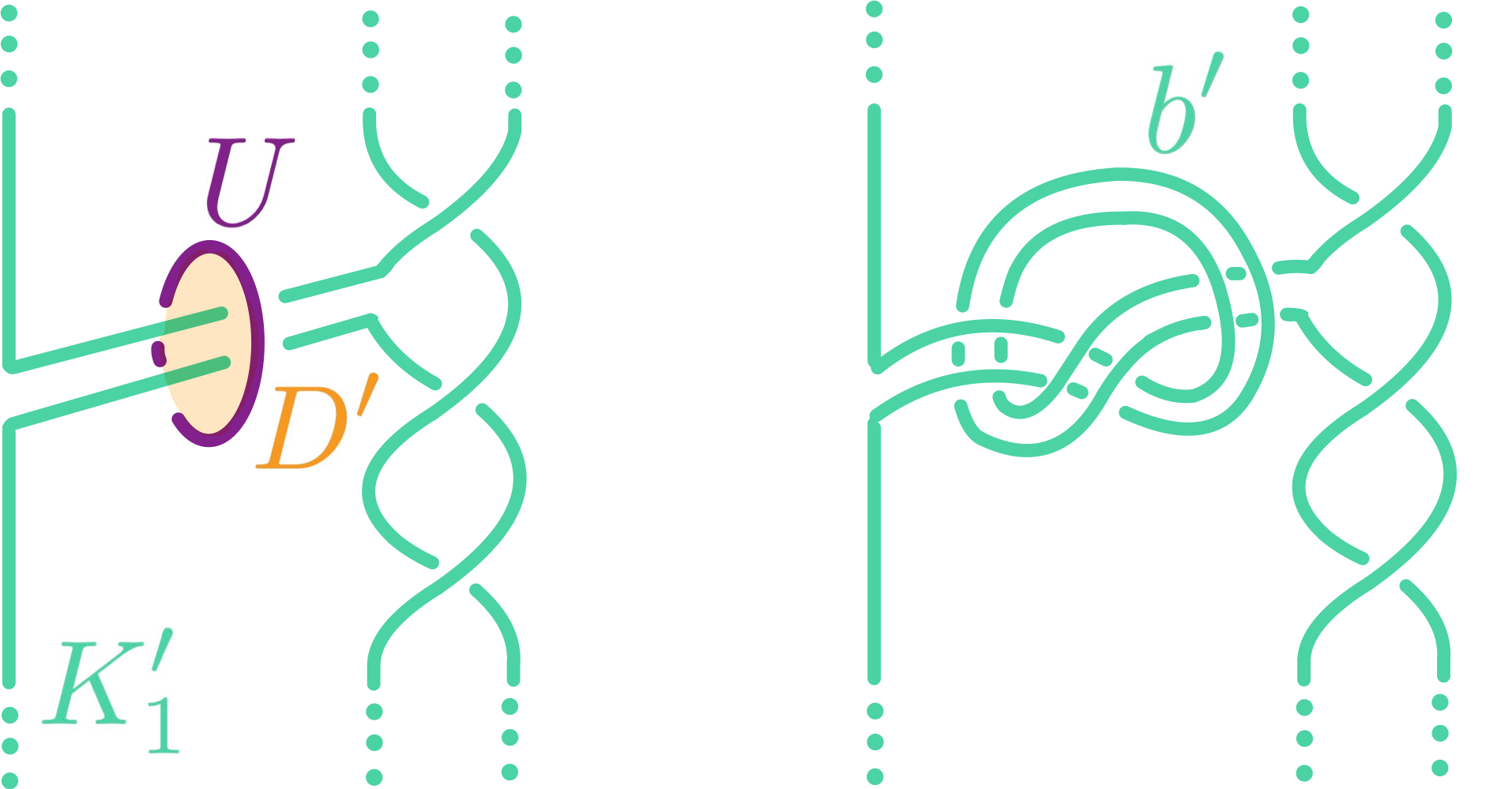} 
 \caption{Left: the unknot $U$ circling the band $b$; Right: The modified band $b'$ obtained by taking a satellite operation.}
 \label{fig:band_complicates}
\end{figure}
 
\noindent
In particular, $g(K_1)=g(K'_1)+g(C_{p_2,q_2}(K_2))=g(K'_1)+(p_2-1)(q_2-1)/2+|q_2|\cdot g(K_2)$. Here, the final equality is due to \cite{Schubert_Genus}. In particular, up to swapping $K_1, K_2$, we see that $g(K_1)>g(K'_1)$. Moreover, the simplicial volume of knot complements is additive under taking connect sums of knots and $||S^3\setminus N(K_1)||\geq ||S^3\setminus N(K'_1)||$. Now let $U'$ be the unknot circling the band $b$ as shown in Figure \ref{fig:band_complicates}. Suppose $U'$ is contractible in $S^3\setminus N(K'_1)$. As before, this implies that $K'_1$ is a composite knot with summands $K_1$ and $C_{p_2,q_2}(K_2)$. In particular, $g(K'_1)\geq g(K_1)>g(K'_1)$, a contradiction. If instead $U'$ is not contractible, then as above we may replace $b$ with some other band such that the simplicial volume satisfies $||S^3\setminus N(K'_1)||>||S^3\setminus N(K_1)||$. This again is a contradiction, showing that $U$ is not contractible in $S^3\setminus N(U\cup L)$.\\

 \begin{figure}[htbp]
 \centering
 \includegraphics[width=0.75\linewidth]{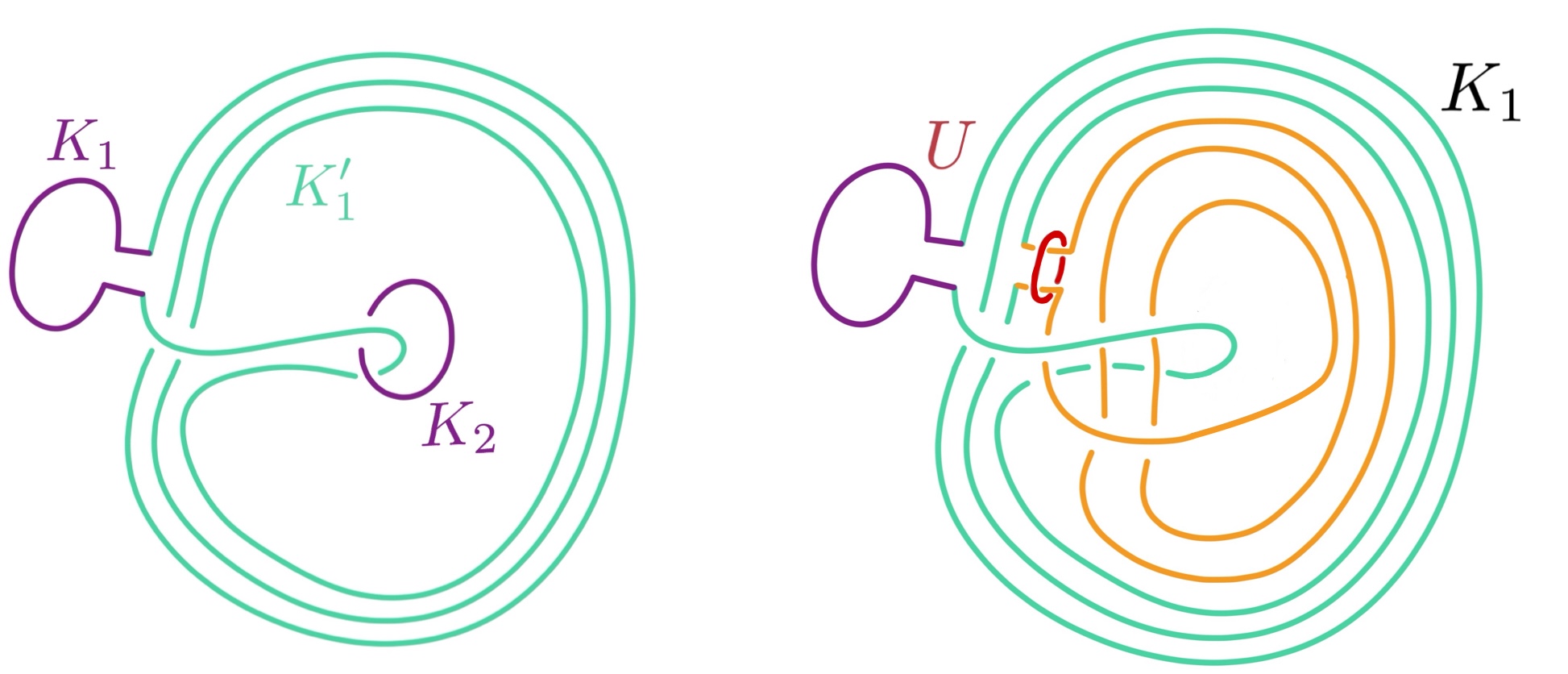} 
 \caption{For $K_1\cup K_2$ the unlink and $p_2=1$, left: $K'_1\cup K_2$ and right: $K_1\cup U$ with $U$ in red.}
 \label{fig:Brunnian}
 \end{figure}

 \begin{figure}[htbp]
 \centering
 \includegraphics[width=0.4\linewidth]{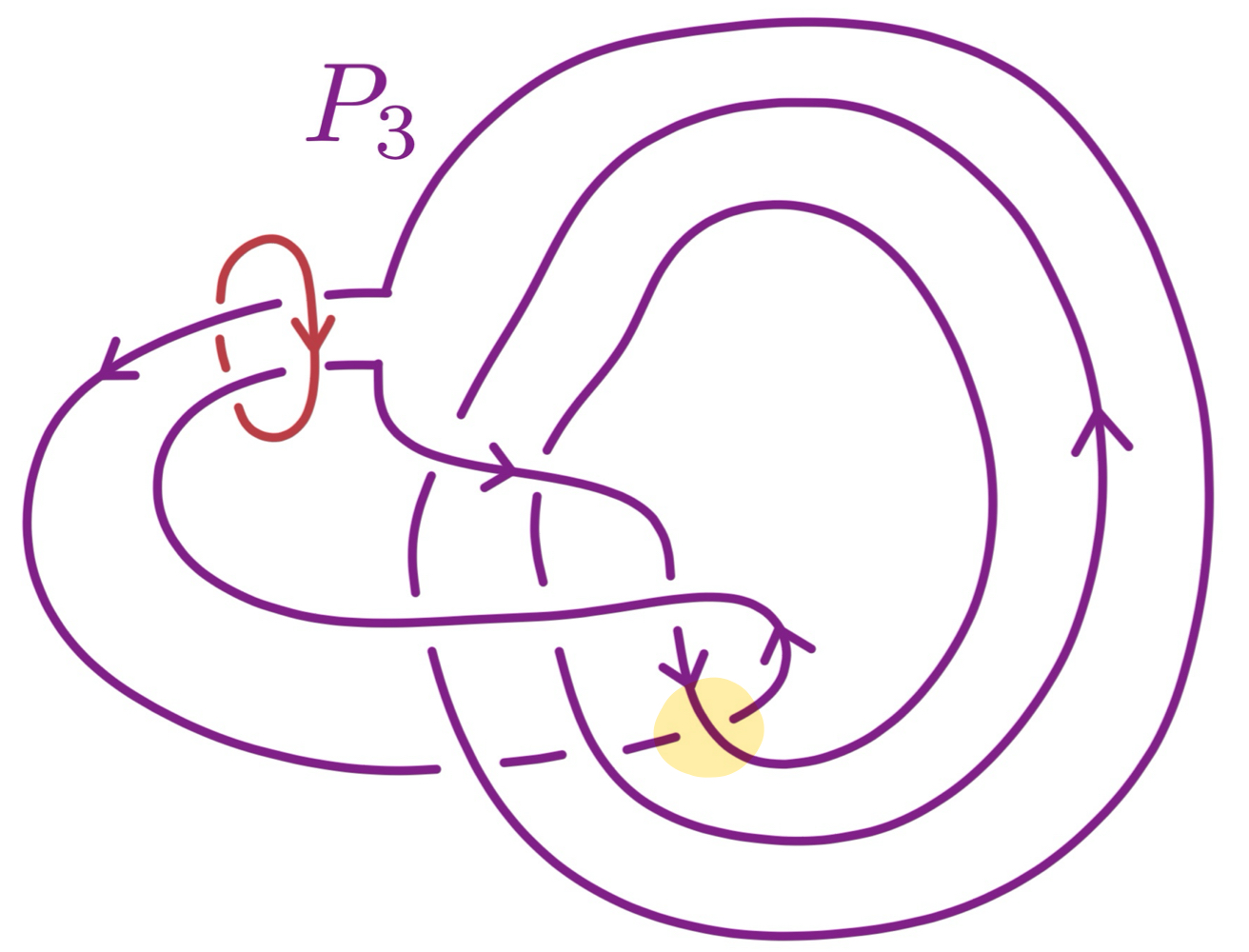} 
 \caption{Simplified $K_1\cup U=P_a$ for $a=3$.}
 \label{fig:Brunnian_simplified}
 \end{figure}

 \noindent
 Case 2. $K_1\cup K_2$ is an unlink and $p_2=1$:\\ 
 In this case, $K_1'\cup K_2$ is as shown on the left in Figure~\ref{fig:Brunnian} and $K_1\cup U$ is as shown on the right in Figure \ref{fig:Brunnian}. We may simplify $K_1\cup U$ to be as shown in Figure \ref{fig:Brunnian_simplified}. Hence, $K_1\cup U$ only depends on the values of $q_2$ and we call the resulting link $P_{q_2}$. Since $K_1$ is an unknot, it is sufficient to show that $P_a$ is not the unlink for any $a\in \mathbb Z\setminus \{0\}$. To do so, let us compute the Jones polynomial $V(P_a)$ with orientations as in Figure \ref{fig:Brunnian_simplified}. Fix $a>1$ consider the crossing of $P_a$ highlighted in Figure \ref{fig:Brunnian_simplified} and resolve this crossing as shown in Figure~\ref{fig:Skein}. Applying the Jones polynomial skein relation yields
 $$V(P_a)=t^2\cdot V(P_{a-1})+(t^{3/2}-t^{1/2})\cdot V(C_3),$$
 for $C_3$ the chain mail link shown in Figure \ref{fig:Skein}. We may compute $V(C_3)=t^{-2}+2+t^2$. Moreover, $P_1$ is a Whitehead link with $V(P_1)=t^{-3/2}\cdot (-1+t-2t^2+t^3-2t^4+t^5)$. This recursive relation implies that for $a\geq 1$: $$V(P_a)=t^{2(a-1)}\cdot V(P_1)+(1+t^2+\ldots+t^{2(a-2)})\cdot (t^{3/2}-t^{1/2})\cdot V(C_3).$$
 In particular, the highest order term of $V(P_a)$ is $t^{2(a-1)}\cdot t^{-3/2}\cdot t^5$. So $P_a$ is not the unlink for $a\geq 0$. An analogous argument shows that for $a\leq - 1$, $P_a$ is not the unlink.

 \begin{figure}[htbp]
 \centering
 \includegraphics[width=0.7\linewidth]{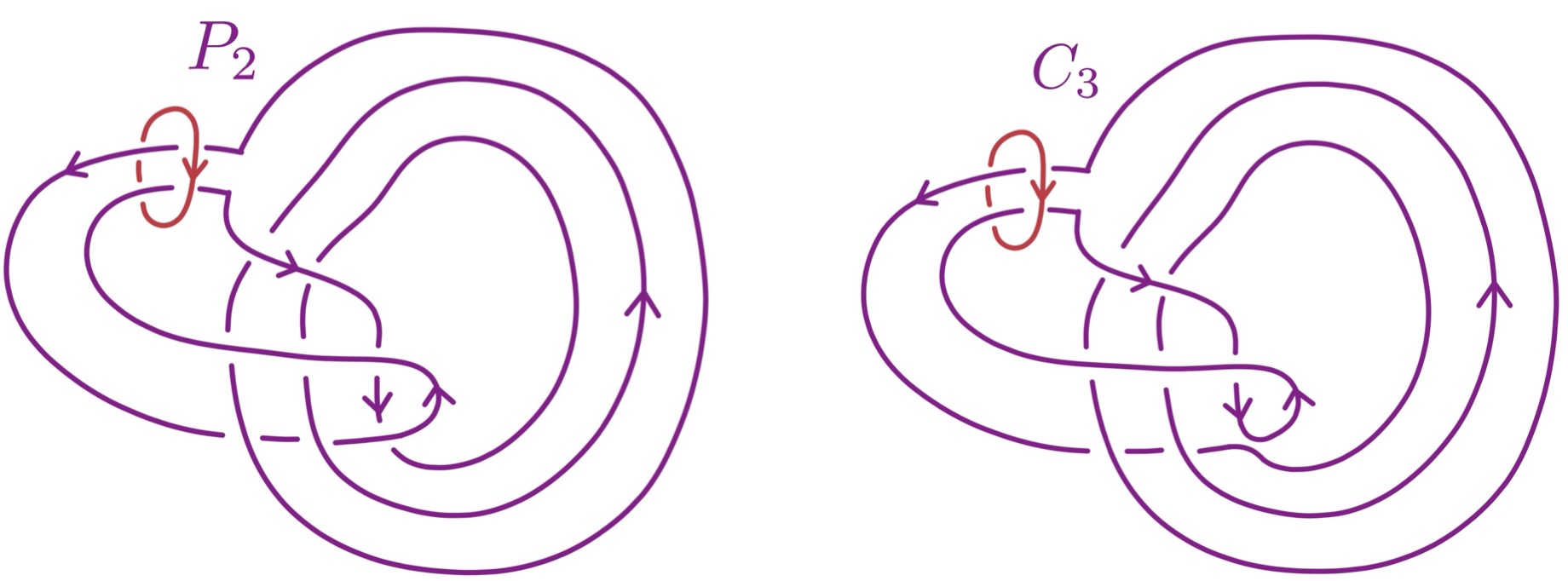} 
 \caption{The resolutions of a crossing of $P_3$.}
 \label{fig:Skein}
 \end{figure}

 \noindent 
Case 3. $K_1\cup K_2$ is an unlink, $p_2\neq 0$ and $q_2=1$: \\
In this case, $K'_1\cup K_2$ is as shown in Figure \ref{fig:p2=1} left and $K_1\cup K_2\cup U$ is as shown in Figure \ref{fig:p2=1} right. So $K_1\cup K_2\cup U$ is obtained as the satellite knot with pattern the Borromean link shown in Figure \ref{fig:final_cases} left and companion $K_2\cup C_{p_2, 1}(K_2)$. Observe that $C_{p_2,1}(K_2)$ and $K_2$ have linking number $p_2$ and therefore $C_{p_2,1}(K_2)$ does not bound a disc in $S^3\setminus N(K_2)$. Moreover, the Borromean link is not split. So $U$ does not bound a disc in $S^3\setminus N(K_1\cup K_2)$.

 \begin{figure}[htbp]
 \centering
 \includegraphics[width=0.85\linewidth]{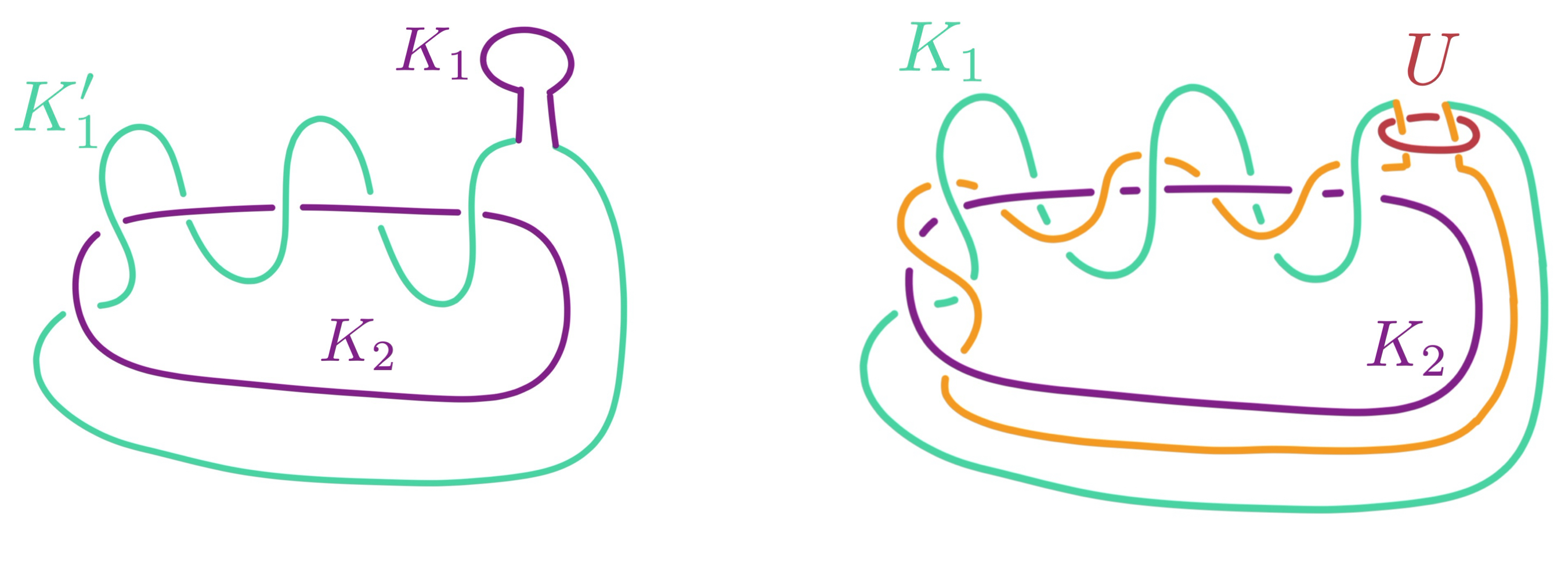} 
 \caption{For $K_1\cup K_2$ the unlink and $q_2=1$, left: $K'_1\cup K_2$ and right: $K_1\cup K_2\cup U$ with $U$ in red.}
 \label{fig:p2=1}
 \end{figure}

\noindent
Case 4. $K_1\cup K_2$ is an unlink and $p_2/q_2=0/1$:\\ We may take the band $b$ so that $K'_1\cup K_2$ is again an unlink. Now take the band $c$ as shown in Figure \ref{fig:final_cases} middle, so that $K_1\cup U$ is as shown in Figure \ref{fig:final_cases} right. This link is hyperbolic, and in particular, $U$ does not bound a disc in $S^3\setminus N(K_1)$. 
 \end{proof}
 
\begin{figure}[htbp]
 \centering
 \includegraphics[width=0.85\linewidth]{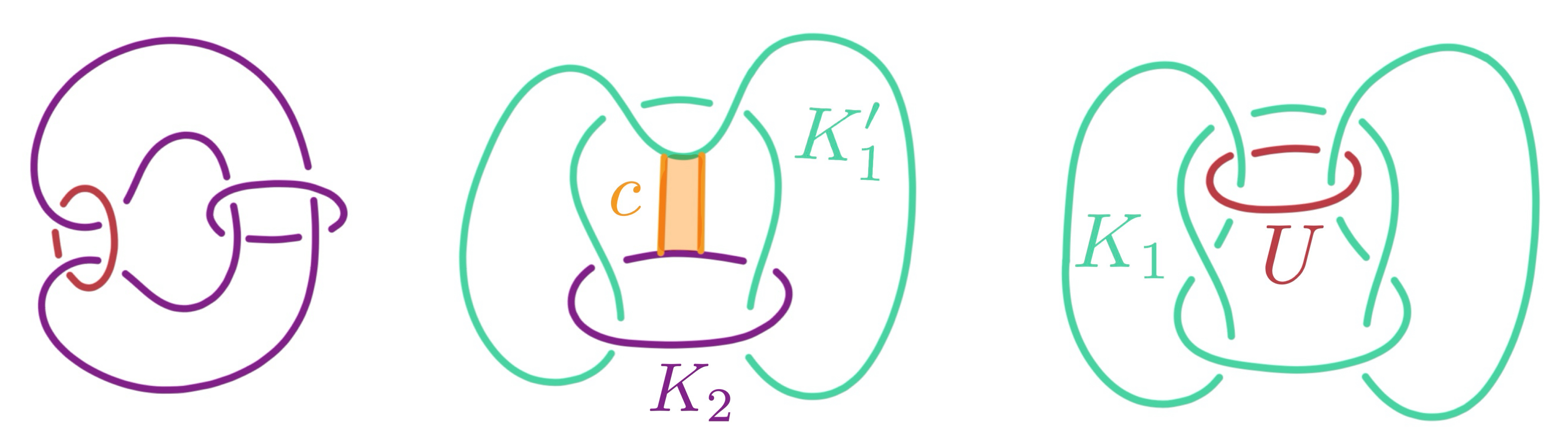} 
 \caption{Left: The Borromean link; middle: The band $c$ in Case 4; right: $K_1\cup U$ in Case 4.}
 \label{fig:final_cases}
 \end{figure}

 \newpage
 \bibliographystyle{hamsalpha}
 \bibliography{lit.bib}

\end{document}